\documentclass{amsart}
\usepackage{amsmath,amssymb}
\usepackage{bbm}
\usepackage{graphicx}
\usepackage{comment}
\usepackage{enumerate}
\usepackage{enumitem}
\usepackage{mathrsfs}
\usepackage{afterpage}
\usepackage[T1,T5]{fontenc}
\usepackage[colorlinks,citecolor=blue,urlcolor=blue,bookmarks=false,hypertexnames=true]{hyperref} 
\usepackage{cleveref}
\usepackage{xcolor}
\usepackage{multicol}
\usepackage{tkz-euclide,amsmath}

\usepackage{etoolbox}

\usepackage{tikz-cd}
\tikzcdset{scale cd/.style={every label/.append style={scale=#1},
		cells={nodes={scale=#1}}}}

\usepackage{tikz}

\def\K{{\mathbbm k}}

\newcommand{\reg}{\operatorname{reg}}

\newcommand{\lcm}{\operatorname{lcm}}
\newcommand{\soc}{\operatorname{socle}}
\newcommand{\depth}{\operatorname{depth}}
\newcommand{\fm}{\mathfrak m}
 \newcommand{\pd}{\operatorname{pd}}
\theoremstyle{plain}

\newtheorem{theorem}{Theorem}[section]
\newtheorem{lemma}[theorem]{Lemma}

\newtheorem{proposition}[theorem]{Proposition}

\theoremstyle{definition}
\newtheorem{definition}[theorem]{Definition}

\newtheorem{example}[theorem]{Example}
\newtheorem{remark}[theorem]{Remark}

\newtheorem{construction}[theorem]{Construction}

\title[]{Linear Resolutions and Linear Quotients in Three Variables}

\author{{Ho\`ai \fontencoding{T5}\selectfont \DJ\`ao}}  
\address[Department of Mathematics]{Virginia Tech}
\email{hoaidao@vt.edu}

\author{Sreehari Suresh-Babu}
\address[Department of Mathematics]{University of Kansas}
\email{sreehari@ku.edu}

\begin{document}
\begin{abstract}
    In this paper, we provide a combinatorial criteria for equigenerated monomial ideals in three variables to have linear resolutions. As a consequence, we prove that in three variables, equigenerated monomial ideals with linear resolutions have linear quotients. 
\end{abstract}
\maketitle
\section{Introduction}
In this paper, we prove the following theorem.
\begin{theorem}
    Let $I\subseteq \K[x,y,z]$ be a monomial ideal generated in a single degree. Then the following statements are equivalent:
    \begin{enumerate}
        \item $I$ has a linear resolution;
        \item $I$ has linear quotients;
        \item Every power of $I$ has a linear resolution.
    \end{enumerate}
\end{theorem}
Our main contribution is to show that if a monomial ideal in $\K[x,y,z]$ has a linear resolution, then it has linear quotients (see \Cref{th:main})\footnote{Via personal communication, we were informed that \Cref{th:main} was independently obtained in the recent preprint \cite{cui} using different methods.}. We give a \emph{geometric} proof. It is based on a combinatorial characterization of linearity of free resolutions of monomial ideals in $\K[x,y,z]$ (see \Cref{th:linreschar}) inspired by a similar characterization of zero-dimensional monomial ideals with almost linear resolutions in \cite[Theorem 4.2]{DaoEisenbud}. 

To put our theorem into context, we briefly recall necessary definitions and similar results. Let $S=\K[x_1,\ldots,x_n]$ be a standard graded polynomial ring over $\K$ and $\fm=(x_1,\ldots,x_n)$ be its homogeneous maximal ideal. If $I\subseteq S$ is a homogeneous ideal generated in a single degree and if all the entries in the differential maps of its minimal free resolution are either zero or linear forms, then we say $I$ has a \emph{linear resolution}. More generally, $I$ is said to \emph{satisfy the condition $N_{d,p}$} if $I$ is generated in degree $d$ and its minimal free resolution is linear up to homological degree $p-1$.

The combinatorial study of the $N_{d,p}$ property of monomial ideals is an active research topic. Fr\"oberg gave a combinatorial description of edge ideals having linear resolutions \cite{Froberg}, and this was later extended to arbitrary $N_{2,p}$ by Eisenbud, Green, Hulek and Popescu \cite{EisenbudEtAl}. Bigdeli, Herzog, and Zaare-Nahandi described monomial ideals satisfying $N_{d,2}$ in terms of their dual graph \cite{IndexOfPowers}. Recently, Dao and Eisenbud provided a characterization of $\fm$-primary monomial ideals satisfying $N_{d,n-1}$ \cite{DaoEisenbud}. But on the other hand, the well-known example of the Stanley-Reisner ideal of a triangulation of the real projective plane $\mathbb{RP}^2$ by Reisner \cite{Reisner} shows that the linearity of the minimal free resolution of an ideal may depend on the characteristic of the base field. Consequently, one cannot always expect a combinatorial characterization of monomial ideals satisfying $N_{d,p}$. 

A \emph{combinatorial} property closely related to linear resolution is the linear quotient property. Recall that a monomial ideal $I$ has \emph{linear quotients} if there exists an order $m_1,\ldots,m_r$ of minimal generators of $I$ such that $(m_1,\ldots,m_{i-1}):m_i$ is generated by variables for $i=2,\ldots,r$. Many known families of ideals have linear quotients; see, for example, \cite{HerzogTakayama, MM, Ficarra}. If $I$ has linear quotients, then $I$ is componentwise linear \cite{SharifanVarbaro, JahanZheng}; see \cite{HerzogHibi} for the definition. In particular, if $I$ is an equigenerated monomial ideal with linear quotients, then $I$ has a linear resolution. The converse is not true, and as far as the authors are aware, Reisner's example mentioned above seems to be the ``smallest'' known monomial counterexample. But on the other hand, the converse is known to be true for quadratic monomial ideals \cite{HerzogHibiZheng}, and monomial ideals in  $\K[x,y]$ \cite{dao2025componentwiselinearidealssums}. For equigenerated monomial ideals in four or five variables, the authors do not know whether the converse holds.

This paper is organized as follows. In \Cref{sec: prelim}, we recall basic definitions and standard results we need. In \Cref{sec:linres}, we present the proofs of our main results.
\subsection*{Acknowledgments} We thank Prof. Hailong Dao for helpful conversations, constant encouragement and his comments on an earlier draft. We are also indebted to him for \Cref{th:linreschar}. The second author acknowledges support from the U.G. Mitchell Scholarship by the Department of Mathematics, University of Kansas. Computations using \texttt{Macaulay2} \cite{M2} have been crucial in our investigation. 


\section{Preliminaries}\label{sec: prelim}

Let $\K$ be a field, and $S=\K[x_1,\ldots, x_n]$ be the standard graded polynomial ring in $n$ variables over $\K$. Let $\fm=(x_1,\ldots,x_n)\subseteq S$ be the homogeneous maximal ideal of $S$. 

Let $I\subseteq S$ be a homogeneous ideal and
 \[
 0\to F_s\to F_{s-1}\to \cdots\to F_1\to F_0
 \]
 be the minimal free resolution of $I$, where $F_i=\oplus_jS(-j)^{\beta_{i,j}(I)}$. The \emph{Castelnuovo-Mumford regularity} of $I$, denoted by $\reg I$, is defined as $\max\{j-i: \beta_{i,j}(I)\neq 0\}$. We say $I$ has a \emph{linear resolution} if $I$ is generated by homogeneous polynomials of some degree $d\geq 0$, and for all $i\geq 1$, $\beta_{i,i+j}(I)=0$ for all $j\neq d$, or equivalently, if $I$ is generated in degree $d$ and $\reg I=d$.

\begin{definition}[{\cite{ConcaHerzog}}]
    We say a homogeneous ideal $I\subseteq S$ has \emph{linear quotients} if there exists a minimal system of homogeneous generators $f_1,\ldots,f_s$ of $I$ such that $(f_1,\ldots,f_{i-1}):f_i$ is generated by linear forms for all $i$.
\end{definition}
An ideal $I\subseteq S$ is a \emph{monomial ideal} if it is generated by monomials. A monomial ideal $I$ has a unique minimal system of monomial generators, which we denote by $\mathcal G(I)$. 

\begin{definition}[{\cite{IndexOfPowers, DaoEisenbud}}]
    Let $I\subseteq S$ be a monomial ideal. The \emph{dual graph} $G_I$ of $I$ is defined as follows: the vertices of $G_I$ are the minimal monomial generators of $I$, and $\{f,g\}$ forms an edge if and only if $\deg \lcm(f,g)=\deg f+1=\deg g+1$.

    If $f,g\in \mathcal G(I)$, then $G_I(f,g)$ is the induced subgraph of $G_I$ with vertex set
    \[
    V(G_I(f,g))=\{h\in G(I): h\text{ divides } \lcm(f,g)\}.
    \]
    We shall denote the dual graph of $\fm^d$ by $\Delta_n(d)$, or simply by $\Delta(d)$ when $n$ is clear from the context. 
\end{definition}

The following result is well-known (see, for example, \cite[Proposition 1.1]{IndexOfPowers} or \cite[Proposition 2.2]{DaoEisenbud}):

\begin{theorem}\label{thm:linearpres}
    Let $I$ be a monomial ideal generated in degree $d$. Then $I$ is linearly presented if and only if $G_I(f,g)$ is connected for all $f,g\in \mathcal{G}(I)$. 
\end{theorem}

\begin{definition}[{\cite[Definition 4.1]{DaoEisenbud}}]
    The \emph{$d$-shadow} of a monomial $f\in S$ is the set of all degree $d$ monomials that divide $f$. We denote it by $S_d(f)$. 
\end{definition}

\begin{example}\label{eg:conv}
    Let $S=\K[x,y,z]$. Our convention of drawing the dual graph of an ideal $I\subseteq S$ generated in degree $d$ is as follows:
    \begin{enumerate}
        \item The dual graph $\Delta(d)$ is drawn as a lattice simplex with the left corner labeled $x^d$, the right corner $y^d$ and the top corner $z^d$.
        \item We draw the dual graph of $G_I$ as an induced subgraph of $\Delta(d)$. 
        \item The generators of $I$ are labeled in blue.
        \item The degree $d$ monomials not in $I$ are labeled in red.
    \end{enumerate}
    The dual graphs of $\fm^3$ and $(x^3, x^2y, xy^2, y^3, x^2z, y^2z)$ are depicted in \Cref{f:d=3}. Note that both ideals are linearly presented by \Cref{thm:linearpres}. Furthermore, the red dots in the second graph form the $3$-shadow of $xyz^3$.
    
    \begin{figure}[hbt]
    \begin{multicols}{2}

    \begin{center}
			\begin{tikzpicture}[scale=0.5][line join = round, line cap = round]

				\draw[-, thick]  (-3,1.5) -- (3,1.5) -- (0,6) -- cycle;
				\draw[-, thick]  (-2,3) -- (-1,1.5); 
				\draw[-, thick]  (-1,4.5) -- (1,4.5);
				\draw[-, thick]   (2,3) -- (1,1.5);
                    \draw[-, thick]   (2,3) -- (-2,3);
                    \draw[-, thick]   (-1,4.5) -- (1,1.5);
                    \draw[-, thick]   (1,4.5) -- (-1,1.5);

				\fill[blue] (-3,1.5) circle (3pt) node[below] {$x^3$};
				\fill[blue] (-1,1.5) circle (3pt) node[below] {$x^2y$};
				
				\fill[blue] (1,1.5) circle (3pt) node[below] {$xy^2$};
				\fill[blue] (3,1.5) circle (3pt) node[below] {$y^3$};
				
				\fill[blue] (-2,3) circle (3pt) node[left] {$x^2z$}; 
                \fill[blue] (0,3) circle (3pt) node[below right] {$xyz$}; 
				\fill[blue] (2,3) circle (3pt) node[right] {$y^2z$};
				
				\fill[blue] (-1,4.5) circle (3pt) node[left] {$xz^2$};
				\fill[blue] (1,4.5) circle (3pt) node[right] {$yz^2$};
				
				\fill[blue] (0,6) circle (3pt) node[above] {$z^3$};

			\end{tikzpicture}
		\end{center}

        \columnbreak 

        \begin{center}
			\begin{tikzpicture}[scale=0.5][line join = round, line cap = round]

				\draw[-, thick]  (-3,1.5) -- (3,1.5) -- (0,6) -- cycle;
				\draw[-, thick]  (-2,3) -- (-1,1.5); 
				\draw[-, thick]  (-1,4.5) -- (1,4.5);
				\draw[-, thick]   (2,3) -- (1,1.5);
                    \draw[-, thick]   (2,3) -- (-2,3);
                    \draw[-, thick]   (-1,4.5) -- (1,1.5);
                    \draw[-, thick]   (1,4.5) -- (-1,1.5);

				\fill[blue] (-3,1.5) circle (3pt) node[below] {$x^3$};
				\fill[blue] (-1,1.5) circle (3pt) node[below] {$x^2y$};
				
				\fill[blue] (1,1.5) circle (3pt) node[below] {$xy^2$};
				\fill[blue] (3,1.5) circle (3pt) node[below] {$y^3$};
				
				\fill[blue] (-2,3) circle (3pt) node[left] {$x^2z$}; 
                \fill[red] (0,3) circle (3pt) node[below right] {$xyz$}; 
				\fill[blue] (2,3) circle (3pt) node[right] {$y^2z$};
				
				\fill[red] (-1,4.5) circle (3pt) node[left] {$xz^2$};
				\fill[red] (1,4.5) circle (3pt) node[right] {$yz^2$};
				
				\fill[red] (0,6) circle (3pt) node[above] {$z^3$};

			\end{tikzpicture}
		\end{center}
        
    \end{multicols}
        \caption{The dual graphs of $(x,y,z)^3$  and $(x^3, x^2y, xy^2, y^3, x^2z, y^2z)$ are on the left and right, respectively. The generators of the ideal are labeled in blue.} 
        \label{f:d=3}
        \end{figure}
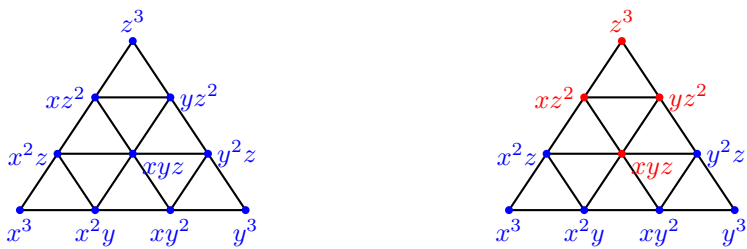
         
\end{example}

 Recall that the \emph{socle} of a homogeneous $S$-module $M$ is the submodule $0:_{M} \fm$, and is denoted by $\soc(M)$.  The depth of $M$ is zero if and only if $\soc(M)\neq 0$. The socle degrees of $M$ can be deduced by looking at the back twists in the minimal free resolution of $M$.
 
\begin{lemma}[{\cite[Lemma 1.3]{soclefits}}]\label{lem: socledeg}
    Let $S=\K[x_1,\ldots, x_n]$, $I\subseteq S$ a homogeneous ideal such that $\depth(S/I)=0$, and let $\mathbb F$ be the minimal graded free resolution of $S/I$. If $F_n=\oplus_{i=1}^r S(-a_i)$, then $\soc(S/I)\cong \oplus_{i=1}^r k(-(a_i-n))$ as graded vector spaces.
\end{lemma}

\section{Monomial Ideals in $\K[x,y,z]$ with a Linear Resolution}\label{sec:linres}
Throughout this section, let $S=\K[x,y,z]$ and $\fm=(x,y,z)$, unless stated otherwise.

\begin{definition}
    Let $I\subseteq S$ be a monomial ideal generated in degree $d$. We say $I$ has a \emph{bad configuration} in $\Delta(d)$ if there exists a nonempty set of degree $d$ monomials $S=\{m_1,\ldots, m_p\}$ not in $I$ such that
    \begin{enumerate}
        \item the set $S$ is the $d$-shadow of a monomial $f$ in $S$. 
        \item The $d$-shadows of $fx,fy$ and $fz$ have a nonempty intersection with $\mathcal{G}(I)$.
    \end{enumerate}
    In this case, we say $f$ induces a bad configuration in $G_I$. 
\end{definition}
\begin{example}
Consider $I=(x^3,x^2y,xy^2,y^3,x^2z,y^2z,xz^2)$ and let $f=xyz$. Then the $3$-shadows of $fx,fy$ and  $fz$ are $\{x^2y,x^2z,xyz\}, \{xy^2,y^2z,xyz\}$ and $\{xz^2, yz^2, xyz\}$ respectively (see \Cref{fig:shadows}). Since each of these shadows intersects $\mathcal G(I)$ nontrivially, $I$ has a bad configuration.
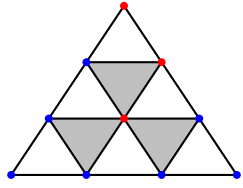
\begin{figure}
    \centering
	\begin{tikzpicture}[scale=0.5][line join = round, line cap = round]
            \fill[lightgray] (-2,3) -- (-1,1.5) -- (0,3);
            \fill[lightgray] (0,3) -- (-1,4.5) -- (1,4.5);
            \fill[lightgray] (0,3) -- (2,3) -- (1,1.5);
            \draw[-, thick]  (-3,1.5) -- (3,1.5) -- (0,6) -- cycle;
            \draw[-, thick]  (-2,3) -- (-1,1.5); 
            \draw[-, thick]  (-1,4.5) -- (1,4.5);
            \draw[-, thick]   (2,3) -- (1,1.5);
            \draw[-, thick]   (2,3) -- (-2,3);
            \draw[-, thick]   (-1,4.5) -- (1,1.5);
            \draw[-, thick]   (1,4.5) -- (-1,1.5);

            \fill[blue] (-3,1.5) circle (3pt) node[below]{} ;
            \fill[blue] (-1,1.5) circle (3pt) node[below]{} ;
            
            \fill[blue] (1,1.5) circle (3pt) node[below]{} ;
            \fill[blue] (3,1.5) circle (3pt) node[below]{} ;
            
            \fill[blue] (-2,3) circle (3pt) node[left]{} ; 
            \fill[red] (0,3) circle (3pt) node[below right]{} ; 
            \fill[blue] (2,3) circle (3pt) node[right]{} ;
            
            \fill[blue] (-1,4.5) circle (3pt) node[left]{} ;
            \fill[red] (1,4.5) circle (3pt) node[right]{} ;
            
            \fill[red] (0,6) circle (3pt) node[above]{};
		\end{tikzpicture}
        \caption{The blue vertices represent the minimal generators of $I=(x^3, x^2y, xy^2, y^3, x^2z, y^2z, xz^2)$. The center red vertex $f=xyz$ induces a bad configuration in $G_I$. The gray upside-down triangles highlight the $d$-shadows of $fx, fy$ and $fz$. }
    \label{fig:shadows}
\end{figure}
\end{example}
\begin{remark}
    It is easy to visually recognize a bad configuration. For example, if $S$ is the $d$-shadow of a monomial $f$, and if the maximum $z$-degree among the monomials in $S$ is $c$, then the $d$-shadow of $fz$ consists of $S$ together with the monomials with $z$-degree $c+1$ that divide $f$. So to find the $d$-shadow of $fz$ on $\Delta(d)$, we just need to look at $(c+1)^{\text{th}}$ level of $\Delta(d)$. Same applies for $fx$ and $fy$ (cf. \Cref{fig:shadows}).
\end{remark}
We now explain the algebraic meaning of a bad configuration. 
\begin{lemma}\label{lem:socle}
    Let $I\subseteq S$ be a monomial ideal generated in degree $d$. Then $I$ has a bad configuration in $\Delta(d)$ if and only if the socle of $S/I$ contains a monomial of degree at least $d$.
\end{lemma}
\begin{proof}
    If $I$ has a bad configuration in $\Delta(d)$, then there exists a nonempty set $M$ of degree $d$ monomials not in $I$ such that $M$ is the $d$-shadow of a monomial $f$ in $I$. Clearly, $f\notin I$ as all the degree $d$ monomials that divide $f$ are outside $I$. Now by assumption, the $d$-shadows of $fx,fy,fz$ intersect $\mathcal G(I)$ nontrivially, so $\fm f\subseteq I$. Hence, $\bar f$, the image of $f$ in $S/I$, belongs to the socle of $S/I$. Finally, the degree of $f$ is at least $d$.

    Conversely, suppose $f$ is a monomial of degree at least $d$ such that $\bar f\in\operatorname{soc}(S/I)$. Since $f\notin I$, the $d$-shadow of $f$ does not intersect $\mathcal{G}(I)$. But $\fm f\subseteq I$, so the $d$-shadows of $fx,fy,fz$ intersect $\mathcal G(I)$ nontrivially. Thus $I$ has a bad configuration in $\Delta(d)$.
\end{proof}
\begin{theorem}\label{th:linreschar}
    Let $I\subset S$ be a monomial ideal generated in degree $d$. Then $I$ has a $d$-linear resolution if and only if $G_I(u,v)$ is connected for all $u,v \in \mathcal G(I)$ and $I$ has no bad configuration in $\Delta(d)$.
\end{theorem}

\begin{proof}
    Suppose first that $I$ has a $d$-linear resolution. Then $I$ has a linear presentation, so $G_I(u,v)$ is connected for all $u,v\in \mathcal G(I)$. If $\pd I=1$, then $I$ has no bad configuration in $\Delta(d)$ by \Cref{lem:socle}, so we are done. If $\pd I=2$, then since $I$ has $d$-linear resolution, the socle of $S/I$ is concentrated in degree $d-1$ (\Cref{lem: socledeg}), so we are again done by \Cref{lem:socle}.

    Conversely, suppose $G_I(u,v)$ is connected for all $u,v\in \mathcal G(I)$ and $I$ has no bad configurations in $\Delta(d)$. If $\pd(I)\leq 1$, then it is clear that $I$ has a $d$-linear resolution. So we suppose $\pd(I)=2$, so $S/I$ has a nonzero socle. We also know that the initial degree of the socle of $S/I$ is at least $d-1$. If $S/I$ has a socle element of degree at least $d$, then $I$ has a bad configuration in $\Delta(d)$ by \Cref{lem:socle}. Thus all the socle elements must have degree $d-1$, so $I$ has a $d$-linear resolution.
\end{proof}

\begin{example}
    Let $I=(x^3,x^2y,xy^2,y^3,x^2z,y^2z,xz^2,yz^2,z^3)$ be the \emph{pinched power ideal} obtained by removing $xyz$ from the generators of $\fm^3$. It is easy to see from the dual graph of $I$ that $I$ is linearly presented. But since $I$ has a bad configuration induced by $xyz$, $I$ does not have a linear resolution. On the other hand, $J=(x^3,x^2y,xy^2,y^3,x^2z,y^2z)$ has a linear resolution (see \Cref{fig:linresexample}). 
    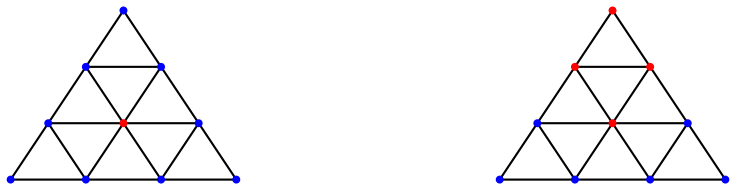
\begin{figure}[hbt]
    \begin{multicols}{2}
    \begin{center}
			\begin{tikzpicture}[scale=0.5][line join = round, line cap = round]

				\draw[-, thick]  (-3,1.5) -- (3,1.5) -- (0,6) -- cycle;
				\draw[-, thick]  (-2,3) -- (-1,1.5); 
				\draw[-, thick]  (-1,4.5) -- (1,4.5);
				\draw[-, thick]   (2,3) -- (1,1.5);
                    \draw[-, thick]   (2,3) -- (-2,3);
                    \draw[-, thick]   (-1,4.5) -- (1,1.5);
                    \draw[-, thick]   (1,4.5) -- (-1,1.5);

				\fill[blue] (-3,1.5) circle (3pt) node[below] {};
				\fill[blue] (-1,1.5) circle (3pt) node[below] {};
				
				\fill[blue] (1,1.5) circle (3pt) node[below] {};
				\fill[blue] (3,1.5) circle (3pt) node[below] {};
				
				\fill[blue] (-2,3) circle (3pt) node[left] {}; 
                \fill[red] (0,3) circle (3pt) node[below right] {}; 
				\fill[blue] (2,3) circle (3pt) node[right] {};
				
				\fill[blue] (-1,4.5) circle (3pt) node[left] {};
				\fill[blue] (1,4.5) circle (3pt) node[right] {};
				
				\fill[blue] (0,6) circle (3pt) node[above] {};

			\end{tikzpicture}
		\end{center}

        \columnbreak 

        \begin{center}
			\begin{tikzpicture}[scale=0.5][line join = round, line cap = round]

				\draw[-, thick]  (-3,1.5) -- (3,1.5) -- (0,6) -- cycle;
				\draw[-, thick]  (-2,3) -- (-1,1.5); 
				\draw[-, thick]  (-1,4.5) -- (1,4.5);
				\draw[-, thick]   (2,3) -- (1,1.5);
                    \draw[-, thick]   (2,3) -- (-2,3);
                    \draw[-, thick]   (-1,4.5) -- (1,1.5);
                    \draw[-, thick]   (1,4.5) -- (-1,1.5);

				\fill[blue] (-3,1.5) circle (3pt) node[below] {};
				\fill[blue] (-1,1.5) circle (3pt) node[below] {};
				
				\fill[blue] (1,1.5) circle (3pt) node[below] {};
				\fill[blue] (3,1.5) circle (3pt) node[below] {};
				
				\fill[blue] (-2,3) circle (3pt) node[left] {}; 
                \fill[red] (0,3) circle (3pt) node[below right] {}; 
				\fill[blue] (2,3) circle (3pt) node[right] {};
				
				\fill[red] (-1,4.5) circle (3pt) node[left] {};
				\fill[red] (1,4.5) circle (3pt) node[right] {};
				
				\fill[red] (0,6) circle (3pt) node[above] {};

			\end{tikzpicture}
		\end{center}
        
    \end{multicols}
        \caption{The dual graphs of $I$ and $J$.}
        \label{fig:linresexample}
        \end{figure}
\end{example}

We will now use \Cref{th:linreschar} to  show that in three variables, equigenerated ideals with linear resolutions have linear quotients. More precisely: 
\begin{theorem} 
    \label{th:main}
    Let $I\subset \K[x,y,z]$ be a monomial ideal generated in degree $d$ with a linear resolution. Then there is an ordering of its minimal generators such that $I$ has linear quotients with respect to the ordering.
\end{theorem}
To prove Theorem \ref{th:main}, we first provide an ordering of the minimal generators of $I$ as follows. 
\begin{construction}[{\bf Tree ordering}] \label{con:tree}  
Let $I\subset \K[x,y,z]$ be a monomial ideal generated in degree $d$ with a linear presentation. We say that a monomial is on the \emph{$c^{\rm th}$ level} if its $z$-degree is $c$. Since $I$ is linearly presented, the minimal generators of $I$ are on consecutive levels.  Assume that minimal generators of $I$ are grouped into $p$ levels with $z$-degrees given by $c,c+1,\cdots,c+(p-1)$. We order the minimal generators of $I$ as follows: 
\begin{itemize}
    \item [\textbf{Step 0.}] For all minimal generators on the $c^{\rm th}$ level, we order  them from left to right (see \Cref{eg:conv} for our conventions), that is, $m_{0,1},m_{0,2},\dots,m_{0,t_0}$, where the $x$-degrees are decreasing, i.e., 
    \[
    \deg_x(m_{0,1})> \deg_x(m_{0,2})> \cdots >\deg_x(m_{0,t_0}).
    \]

    \item [\textbf{Step 1.}] Inductively, for $0<i<p$, generators on the ${(c+i)}^{\rm th}$ level are ordered as follows: denote by $m_{i,1}$ the minimal generator of $I$ with the highest $x$-degree such that \emph{it is joined to $m_{i-1,j}$} by an edge for some $j$. For the generators on the ${(c+i)}^{\rm th}$ level such that their $x$-degrees are less than the $x$-degree of $m_{i,1}$, we order them from left to right. Then we order the remaining generators on the ${(c+i)}^{\rm th}$ level from right to left, that is, the $x$-degrees are increasing.
\end{itemize}
We call this ordering the \emph{tree ordering} of minimal generators of $I$. 
\end{construction}
\begin{example}
    The minimal generators of $\fm^3$ listed in tree ordering are as follows: $x^3, x^2y, xy^2,y^3, x^2z, xyz, y^2z, xz^2, yz^2, z^3$ (see \Cref{f:d=3}). The blue vertices in \Cref{fig:deg4eg} listed in tree ordering are $xy^3, xy^2z, y^3z, x^2yz, x^3z, x^2z^2, y^2z^2$. Note that the ideal $I$ generated by these monomials has linear resolution.
    \begin{figure}[hbt]
        \centering
        \begin{tikzpicture}[scale=0.6, line join=round, line cap=round]
          \draw[thick] (-4,2) -- (4,2) -- (0,8) -- cycle;
          \draw[thick] (-3,3.5) -- (3,3.5);
          \draw[thick] (-2,5) -- (2,5);
          \draw[thick] (-1,6.5) -- (1,6.5);
          \draw[thick] (-3,3.5) -- (-2,2);
          \draw[thick] (-2,5)   -- (0,2);
          \draw[thick] (-1,6.5) -- (2,2);
          \draw[thick] (3,3.5) -- (2,2);
          \draw[thick] (2,5)   -- (0,2);
          \draw[thick] (1,6.5) -- (-2,2);
          \fill[red] (-4,2) circle (3pt) node[below] {$x^4$};
          \fill[red] (-2,2) circle (3pt) node[below] {$x^3y$};
          \fill[red] (0,2)  circle (3pt) node[below] {$x^2y^2$};
          \fill[blue] (2,2)  circle (3pt) node[below] {$xy^3$};
          \fill[red] (4,2)  circle (3pt) node[below] {$y^4$};
          \fill[blue] (-3,3.5) circle (3pt) node[left] {$x^3z$};
          \fill[blue] (-1,3.5) circle (3pt) node[below] {$x^2yz$};
          \fill[blue] (1,3.5) circle (3pt) node[below] {$xy^2z$};
          \fill[blue] (3,3.5) circle (3pt) node[right] {$y^3z$};
          \fill[blue] (-2,5) circle (3pt) node[left] {$x^2z^2$};
          \fill[red] (0,5)  circle (3pt) node[below] {$xyz^2$};
          \fill[blue] (2,5)  circle (3pt) node[right] {$y^2z^2$};
          \fill[red] (-1,6.5) circle (3pt) node[left] {$xz^3$};
          \fill[red] (1,6.5) circle (3pt) node[right] {$yz^3$};
          \fill[red] (0,8) circle (3pt) node[above] {$z^4$};
        \end{tikzpicture}
        \caption{The dual graph of $I=(xy^3, xy^2z, y^3z, x^2yz, x^3z, x^2z^2, y^2z^2)$.}
        \label{fig:deg4eg}
    \end{figure}
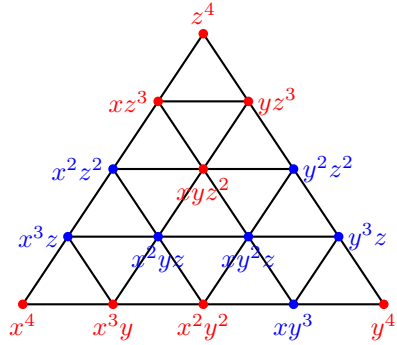
\end{example}

The following result is a crucial ingredient in the proof of Theorem \ref{th:main}. 

\begin{proposition} \label{prop:linearquotient}
    Let $I\subseteq \K[x_1,\ldots, x_n]$ be a monomial ideal generated in degree $d$, and suppose $\mathcal{G}(I)=\{m_1,\ldots, m_r\}$. Then the following are equivalent:
    \begin{enumerate}
        \item $I$ has linear quotients with respect to $m_1,\ldots, m_r$.
        \item The ideals $I_j=(m_1,\ldots, m_j)$ have linear resolution for $j=1,\ldots, r$.
        \item The ideals $I_j=(m_1,\ldots, m_j)$ are linearly presented for $j=1,\ldots, r$. 
    \end{enumerate}
\end{proposition}

\begin{proof}
    (1) implies (2) is \cite[Lemma 2.5]{Herzog-Moradi-Rahimbeigi}. (2) implies (3) is trivial. (3) implies (1) is {\cite[Proposition 5.6]{DaoEisenbud}} or {\cite[Lemma 3.1]{DaoDoolittleLyle}}. 
\end{proof}

For the remainder of the section, let $I\subseteq \K[x,y,z]$ be a monomial ideal generated in degree $d$ having a linear resolution, and suppose $\mathcal{G}(I)=\{m_1,\ldots, m_r\}$ with $m_1,\ldots, m_r$ listed in tree order.

To prove Theorem \ref{th:main}, we will also make use of the following two lemmas.                                                 
\begin{lemma}\label{lem:sameLevel}
    Let $I_j=(m_1,\ldots, m_j)$ for $2\leq j \leq r$. If $u,v\in \mathcal G(I_j)$ are on the same level, then $G_{I_j}(u,v)$ is connected.
\end{lemma}

\begin{proof}
    By \Cref{th:linreschar}, it suffices to show that $G_I(u,v)=G_{I_j}(u,v)$. Suppose $u$ and $v$ are on the  $c^{\rm th}$ level, and let $m\in G_I(u,v)$. Then $\deg_z(m)\leq c$. If $\deg_z(m)< c$, then $m\in I_j$ by construction. If $\deg_z(m)=c$, then $m$ comes in between $u$ and $v$ in tree ordering, and hence $m\in I_j$ again by construction.    
\end{proof}
\begin{lemma}\label{lem:disconnected}
    If $I_j$ is linearly presented for all $j<\ell$ and $I_\ell$ is not linearly presented, then $G_{I_\ell}(m_i,m_\ell)$ is disconnected for some $i\leq \ell-1$. 
\end{lemma}

\begin{proof}
 Suppose on the contrary that $G_{I_\ell}(u,m_\ell)$ is connected for all $u\in \mathcal G(I_{\ell-1})$. Let $u,v \in \mathcal G(I_\ell)$ and consider $G_{I_\ell}(u,v)$. If $V(G_{I_\ell}(u,v))=V(G_{I_{\ell-1}}(u,v))$, then $G_{I_\ell}(u,v)=G_{I_{\ell-1}}(u,v)$ is connected. Otherwise, $V(G_{I_\ell}(u,v))=V(G_{I_{\ell-1}}(u,v))\cup \{m_\ell\}$, which implies $m_\ell | \lcm (u,v)$. Therefore, $\lcm(u,m_\ell)| \lcm(u,v)$, and hence $G_{I_\ell}(u,m_\ell)\subseteq G_{I_\ell}(u,v)$. As $G_{I_\ell}(u,m_\ell)$ is connected, there is a path in $G_{I_\ell}(u,v)$ from $u$ to $m_\ell$. But this implies $m_\ell$ is connected to every other vertex in $G_{I_\ell}(u,v)$, so $G_{I_\ell}(u,v)$ is connected.  Thus, $G_{I_\ell}(u,v)$ is connected for all $u,v\in \mathcal G(I_\ell)$, so $I_\ell$ is linearly presented, a contradiction.
\end{proof}
Now we are able to prove Theorem \ref{th:main}.

\begin{proof}[Proof of Theorem \ref{th:main}]
     We order the minimal generators of $I$ as in Construction \ref{con:tree}, and then by renumbering, assume that $I=(m_1,\dots,m_r)$. By \Cref{prop:linearquotient}, it suffices to show that $I_j=(m_1,\ldots,m_j)$ is linearly presented for all $j$. 
     By way of contradiction, suppose not. Let $\ell$ be the smallest index such that $I_\ell$ is \emph{not} linearly presented. \Cref{lem:disconnected} implies that $G_{I_\ell}(m_i,m_{\ell})$ is disconnected for some $i\leq \ell-1$. Let $k$ be the greatest among such $i$. 

    \noindent\textbf{Claim 1.} We claim that $m_s\notin G_{I_\ell}(m_k,m_\ell)$ for all $k<s<\ell$. Indeed, if $m_{s}\in G_{I_\ell}(m_k,m_\ell)$, then $G_{I_\ell}(m_k,m_{s})\subseteq G_{I_\ell}(m_k,m_\ell)$. Since $G_{I_s}(m_k,m_s)$ and $G_{I_\ell}(m_s,m_\ell)$ are connected, there is a path from $m_k$ to $m_\ell$ via $m_{s}$ in $G_{I_\ell}(m_k,m_\ell)$, a contradiction.
    
    Assume that $m_k=x^{a_1}y^{b_1}z^{c_1}$ and $m_\ell=x^{a_2}y^{b_2}z^{c_2}$. As $k<\ell$, $c_1\leq c_2$. \Cref{lem:sameLevel} implies that $c_1<c_2$. Next, we prove that $c_2-c_1=1$. Suppose that $c_2-c_1>1$. Since $G_{I}(m_k,m_{\ell})$ is connected, there exists a path connecting $m_k$ and $m_{\ell}$, so there is some $m_s=x^{a}y^{b}z^{c_2-1}$ in $G_I(m_k, m_\ell)$. We have $k<s<\ell$, so $m_s\in I_{\ell}$, and thus $m_s\in G_{I_{\ell}}(m_k,m_{\ell})$, a contradiction to Claim 1. Thus $c_2-c_1=1$ and it follows that $a_1\neq a_2$ and $b_1\neq b_2$.

    So $m_k=x^{a_1}y^{b_1}z^{c_1}$ and $m_\ell=x^{a_2}y^{b_2}z^{c_1+1}$ with either $a_1>a_2$ or $a_1<a_2$. Consider two cases as follows.

    \noindent\textbf{Case 1.} $a_1>a_2$

    We have $a_1+b_1+c_1=a_2+b_2+c_1+1$, so $b_2-b_1=a_1-a_2-1 \geq 0$. This implies that $b_2>b_1$, so $\lcm (m_k,m_{\ell})=x^{a_1}y^{b_2}z^{c_1+1}$.  
        Since $G_I(m_k,m_{\ell})$ is connected, $f=m_\ell\cdot\left({x}/{y}\right)$ or $g=m_\ell\cdot\left({x}/{z}\right)$ is in $I$ (see Figure \ref{f:case1}). 
        \begin{figure}[hbt]
            \centering
            \begin{tikzpicture}[scale=0.5]
            \draw [thick] (-6,0) -- (6,0); 
            \draw [thick] (-6,-1) -- (6,-1); 
            \draw [thick] (-5,2) -- (1,-10);
            \draw [thick] (-1,-10) -- (5,2);
            \fill[blue] (-3.5,-1) circle (3pt); 
            \fill[blue] (-3.4,-1.3) circle (0pt) node[left] {$m_k$}; 
            \fill[blue] (4,0) circle (3pt); 
            \fill[blue] (3.9,-0.3) circle (0pt) node[right] {$m_\ell$}; 
            
            \fill[] (3,0) circle (3pt) node[above] {$f$};  
            \fill[red] (3.5,-1) circle (3pt) node[below] {$g$};  
            \fill[red] (-1.5,-1) circle (3pt) node[below] {$g'$};
            \fill[] (0.5,-3) circle (3pt) node[below] {$m_t$};
            \fill[] (0.5,-1) circle (3pt) node[below] {$m$};
            \fill[] (4.5,-1) circle (3pt) node[below] {$h$};
            \end{tikzpicture}
            \caption{Case 1.}
            \label{f:case1}
        \end{figure}
        
        If $g\in I$, then $g\in I_{\ell-1}$. But then the connected graph $G_{I_\ell}(m_k,g)$ is contained in $G_{I_\ell}(m_k,m_\ell)$, and this implies  $G_{I_\ell}(m_k,m_\ell)$ is connected, a contradiction. Thus, $g=x^{a_2+1}y^{b_2}z^{c_1} \not \in I$ and $f= x^{a_2+1}y^{b_2-1}z^{c_2}=m_i \in I$ for some $i$. Next, we prove that $f\not\in I_\ell$. If $f\in I_\ell$, then $f=m_s \in G_{I_\ell}(m_k,m_\ell)$ for some $k<s<\ell$, a contradiction to Claim 1. Hence, $f\in I\setminus I_\ell$.
        

        We now show that for all vertices of the form $m_s=x^py^qz^{c_2} \in I_\ell$ with $k<s<\ell$, we must have $p\leq a_2$. Suppose not, i.e., there is a vertex, say  $m_s=x^py^qz^{c_2} \in I_\ell$ for some $k<s<\ell$, such that $p > a_2$. We have $q < b_2$, so $\lcm (m_s,m_\ell)=x^{p}y^{b_2}z^{c_2}$. Since $f,g\not\in I_\ell$, we obtain that $m_\ell$ is an isolated vertex of $G_{I_\ell}(m_s,m_\ell)$, a contradiction since our choice of $k$ implies that $G_{I_\ell}(m_s,m_\ell)$ is connected.   

        Next, we claim that $h=m_\ell({y}/{z}) \in I_\ell$. If $m_\ell$ is the first generator of $I$ on the $c_2^{\rm th}$ level, then $m_\ell$ must be adjacent to a vertex from the lower level, and since $g\notin I$, $h$ must be in $I$. Otherwise, $m_{\ell-1}=x^py^qz^c_2$ for some $p \leq a_2$. If $h\notin I_\ell$, then as $G_{I_\ell}(m_{\ell-1}, m_{\ell})$ is connected, we have $m_{\ell-1}=m_\ell(y/x)$. But then in $G_{I_{\ell-1}}(m_k,m_{\ell-1})$, $m_{\ell-1}$ is an isolated vertex, which contradicts the fact that $I_{\ell-1}$ is linearly presented. Thus, $h\in I_\ell$. 

        Let $m_q\in G_{I_\ell}(m_k,m_\ell)$ be the generator on the $c_1^{\rm th}$ level with the smallest $x$-degree such that $g'=m_q(y/x)\notin I_\ell$. Then none of the monomials in the segment $[g,g']$ are in $I$, that is, $g', g'({y}/{x}),g'({y}/{x})^2, \ldots, g$ are not in $I$. Consider
        $u=\lcm(g,g')$ whose $z$-degree is $c_1$. We complete the proof of Case 1 by showing that $u$ induces a bad configuration in $G_I$. Let $S$ be the $d$-shadow of $u$. We claim that $S$ consists of monomials not in $I$. On the contrary,
        suppose that there exists some $m_t=x^{\alpha}y^{\beta} z^{\gamma}\in \mathcal{G}(I)$ dividing $u$ for some $t<\ell$. Because of the choice of $u$, $\gamma<c_1$. Since $I$ is linearly presented, $G_I(m_t,m_\ell)$ is connected. Then there is a vertex $m$ of $G_I(m_t,m_\ell)$ with $z$-degree $c_1$.  It is not difficult to see that $\lcm(m_t,m_\ell)=x^{\alpha}y^{b_2}z^{c_2}$, so $m$ must be in the segment $[g,g']$, a contradiction. This proves the claim. However, the $d$-shadow of $ux$ contains $m_q\in I$, that of $uy$ contains $h\in I$, and that of $uz$ contains $m_\ell \in I$. Thus $u$ induces a bad configuration in $G_I$, a contradiction since $I$ has a linear resolution. 

        \noindent\textbf{Case 2.} $a_1<a_2$

        \begin{figure}[h!]
            \centering
            \begin{tikzpicture}[scale=0.5]
            \draw [thick] (-6,0) -- (6,0); 
            \draw [thick] (-6,-1) -- (6,-1); 
            \draw [thick] (-5,2) -- (1,-10);
            \draw [thick] (-1,-10) -- (5,2);
            \fill[red] (-3.5,-1) circle (3pt); 
            \fill[red] (-3.4,-1.3) circle (0pt) node[left] {$g$}; 
            \fill[] (-4.5,-1) circle (3pt) node[below] {$h$}; 
            \fill[blue] (-4,0) circle (3pt); 
            \fill[blue] (-5,-0.3) circle (0pt) node[right] {$m_\ell$}; 
            
            \fill[] (-3,0) circle (3pt) node[above] {$f$}; 
            \fill[red] (1,-1) circle (3pt) node[below] {$g'$};
            
            \fill[blue] (3.5,-1) circle (3pt) node[below] {};  
            \fill[blue] (3.9,-1) circle (0pt) node[below] {$m_k$}; 
            \end{tikzpicture}
            \caption{Case 2.}
            \label{f:case2}
        \end{figure}
    This case is similar to Case 1 and only requires minor modifications. We start by noticing that $\lcm(m_k,m_\ell)=x^{a_2}y^{b_1}z^{c_1+1}$ and $b_1> b_2+1$. As above,  $g=m_\ell(y/z)=x^{a_2}y^{b_2+1}z^{c_1}\notin I$ and $f=m_\ell(y/x)\in I\setminus I_\ell$ (see \Cref{f:case2}).

    Now suppose there exists some $m_s=x^ay^bz^{c_2}\in \mathcal G(I_\ell)$ with $a<a_2$. Then $\lcm(m_s,m_\ell)=x^{a_2}y^bz^{c_2}$. Since $f,g\notin I_\ell$, $m_\ell$ is an isolated vertex in $G_{I_\ell}(m_s,m_\ell)$, a contradiction to the connectedness of $G_{I_\ell}(m_s,m_\ell)$. Thus, for all $m_s=x^ay^bz^{c_2}\in I_\ell$, we must have $a\geq a_2$. 

    Next, we claim that $h=m_\ell(x/z)\in I_\ell$. If $m_\ell$ is the first generator of $I$ on the $c_2^{\rm th}$ level, then $m_\ell$ must be adjacent to a vertex from the lower level, and since $g\notin I$, $h$ must be in $I$. Otherwise, $m_{\ell-1}$ is on the same level as $m_\ell$, so $\deg_x(m_{\ell-1})\geq a_2$. If $h\notin I_\ell$, then as $G_{I_\ell}(m_{\ell-1}, m_{\ell})$ is connected, we have $m_{\ell-1}=m_\ell(x/y)$. But then in $G_{I_{\ell-1}}(m_k,m_{\ell-1})$, $m_{\ell-1}$ is an isolated vertex, which contradicts the fact that $I_{\ell-1}$ is linearly presented. Thus, $h\in I_\ell$. 

    Let $m_q\in G_{I_\ell}(m_k,m_\ell)$ be the generator on the $c_1^{\rm th}$ level having the largest $x$-degree such that $g'=m_q(x/y)\notin I_\ell$. Then $g, g(y/x), g(y/x)^2, \ldots, g'$ do not belong to $I$. Let $S$ be the $d$-shadow of $u=\lcm(g,g')$. Suppose there exists some $m_t\in \mathcal G(I_\ell)\cap S$. Consider the connected graph $G_I(m_t, m_\ell)$. Since $\deg_z(m_t)<c_1$, there exists some $m\in \mathcal G(I_\ell)$ such that $\deg_z(m)=c_1$. Since the $x$-degree of $\lcm(m_t,m_\ell)$ is $a_2$, $\deg_x(m)\leq a_2=\deg_x(g)$. Similarly, $\deg_y(m)\leq \deg_y(\lcm(m_t,m_\ell))=\deg_y(m_t)\leq \deg_y(g')$. Hence, $m\in S$. But all elements in $S$ with $z$-degree equals $c_1$ are outside of $I$. Thus, $S$ consists of degree $d$ monomials not belonging to $I$. However, the $d$-shadow of $ux$ contains $h\in I$, that of $uy$ contains $m_t\in I$ and that of $uz$ contains $f\in I$. This implies $I$ has a bad configuration in $\Delta(d)$, a contradiction.  
\end{proof}
We finally have the theorem that was stated in the introduction.
\begin{theorem}
    Let $I\subseteq \K[x,y,z]$ be a monomial ideal generated in a single degree. Then the following are equivalent:
    \begin{enumerate}
        \item $I$ has a linear resolution;
        \item $I$ has linear quotients;
        \item Every power of $I$ has a linear resolution.
    \end{enumerate}
\end{theorem}
\begin{proof}
    The equivalence of (1) and (2) is established by \Cref{th:main} and \cite[Lemma 4.1]{ConcaHerzog}. (3) implies (1) is clear. (1) implies (3) follows from \cite[Corollary 7.9]{EHU}. 
\end{proof}

\bibliographystyle{amsplain}
\bibliography{references}

\end{document}